\documentclass{article}
\usepackage[margin=1in]{geometry}
\usepackage{amsmath} 
\usepackage{amsthm} 
\usepackage{microtype}
\usepackage{amssymb}
\usepackage{hyperref}

\newtheorem{theorem}{Theorem}
\newtheorem{corollary}[theorem]{Corollary}
\newtheorem{lemma}[theorem]{Lemma}
\newtheorem{prop}[theorem]{Proposition}
\newtheorem*{conjecture*}{Conjecture}
\theoremstyle{remark}
\newtheorem*{remark}{Remark}

\title{Counting Reachable Positions in a Simplified Model of Connect Four}
\author{Alistair Hartley Folster\\
\small Columbus State Community College}
\date{}

\begin{document}

\maketitle

\begin{center}
\small
\textit{Author's accepted manuscript. Published in the Journal of Combinatorial Mathematics and Combinatorial Computing, Volume 131, pp.~57--65, 12 June 2026. The published version is available at \url{https://doi.org/10.61091/jcmcc131-04}.}
\end{center}

\begin{abstract}
By eliminating the win condition in the game of Connect Four and extending the board to infinite height, a rich state space of positions is obtained. We investigate the number of positions reachable on an $n$-column board after $k$ color-alternating moves. For fixed $k$ we demonstrate polynomiality, derive a partial formula for the polynomial coefficients, and precisely characterize the asymptotic behavior as $n \to \infty$. We then turn our attention to the fixed-$n$ case and show that, under a natural addition operation, positions reachable in an even number of moves form a monoid with a highly symmetric finite generating set; by examining certain free submonoids, we bound the exponential growth rate as $k \to \infty$.

\end{abstract}

\section{Introduction}

For the standard $7\times 6$ board, the number of legal Connect Four positions at various ply counts appears in the On-Line Encyclopedia of Integer Sequences (OEIS) entry A212693 \cite{OEIS}. Edelkamp and Kissmann \cite{Edelkamp_Kissmann} studied the complexity of BDD representations of the state space as a function of board size, both with and without the game's win condition enforced. However, little work has been done on the combinatorial theory of the state space. As noted by Allis \cite{Allis} in his work that solved the game as a first-player win, while upper bounds on the number of legal positions are easily established, excluding illegal positions is highly nontrivial owing to the twin constraints of the win condition and the turn-based nature of play.

This paper examines the consequences of the latter constraint. In approaching the combinatorics of Connect Four, it is natural to isolate the core mechanics of token stacking and color alternation, which produces a combinatorial object of independent appeal. To this end, we disregard the win condition and consider an infinitely tall board.

Let a \emph{column} be an infinite vertical stack of cells and a \emph{board} be a finite sequence of columns. On a given board, define a \emph{position} to be a placement of finitely many tokens in cells such that each token is either red or yellow, no two tokens occupy the same cell, and the occupied cells in each column form a contiguous block starting from the bottom. 

A \emph{reachable} position is one which may be reached from an empty board by a finite sequence of moves, where each move consists of the insertion of a token into the lowest empty cell in some column, with the inserted tokens alternating globally between red and yellow from move to move. (We adopt the convention that a red token is placed on the first move.) On an $n$-column board, we denote the number of positions reachable in exactly $k$ moves by $F(n, k)$.

The initial values of $F(2, k)$ and $F(3, k)$ appear in OEIS entries A320452 and A320731 \cite{OEIS}; the definitions given therein are equivalent to our own up to terminological differences, and we have independently reproduced the values by BFS of the state space. This paper establishes several basic results about the structural underpinnings, algebraic interpretation, and asymptotic behavior of these position counts.

\section{Some Additional Terminology and a Preliminary Upper Bound}
\label{sec:upper-bound}

We refer to the number of tokens in a position as its \emph{size}. On a given board, a position of size $k$ can be uniquely specified by distributing $k$ indistinguishable tokens among the columns and then coloring every token either red or yellow; this gives $\binom{n+k-1}{k}2^k$ size-$k$ positions on an $n$-column board. 

For a position of size $k$ to be reachable, it must contain exactly $\lceil k/2 \rceil$ red tokens. We will call positions that fulfill this condition \emph{color-balanced}. This yields

\begin{equation}
\label{eq:color-balanced}
F(n, k) \leq
\binom{n+k-1}{k}\binom{k}{\lceil \frac{k}{2} \rceil}.
\end{equation}

\section{Behavior of \texorpdfstring{$F(n, k)$}{F(n, k)} for Fixed \texorpdfstring{$k$}{k}}

On a board with at least $i$ columns, fix any $i$ columns. The number of reachable positions of size $k$ with precisely these columns nonempty depends only on $k$ and $i$; the placement of the columns within the board is irrelevant. Let $T_{k, i}$ denote this number.

\begin{theorem}
\label{thm:polynomiality}
 For each $k$, $F(n, k)$ is given by a polynomial in $n$ of degree $k$. 
\end{theorem}

\begin{proof}
On an $n$-column board, there are $T_{k, i}\,\binom{n}{i}$ reachable positions of size $k$ with exactly $i$ columns nonempty, where $\binom{n}{i}$ counts choices of $i$ columns and $T_{k, i}$ counts size-$k$ positions with just the chosen columns nonempty.

Since the number $i$ of nonempty columns after $k$ moves ($k > 0$) satisfies $1 \leq i \leq k$, we can write
\begin{equation}
\label{eq:F-expansion}
F(n, k) = \sum_{i = 1}^{k} T_{k, i}\,\binom{n}{i}.
\end{equation}
(We interpret $\binom{n}{i} = 0$ for $i > n$, so the formula remains valid for all $k$.)

Since $T_{k, k} \neq 0$, this is a polynomial in $n$ of degree $k$.
\end{proof}

Understanding the coefficients $T_{k, i}$ would yield insight into $F(n, k)$. Of course, $T_{k, 1} = 1$ for all $k$ as there is only one reachable position of a given size using only a single specific column. The next theorem will enable us to additionally obtain a formula for all values of $T_{k, i}$  with $i > k/2$. 

Let the \emph{crown} of a position be the set of tokens with no tokens above them in the same column. 

\begin{lemma}
\label{lem:removal}
If removing a token from the crown of a color-balanced position $P$ yields a reachable position, then $P$ is reachable.
\end{lemma}

\begin{proof} 
$P$ is obtained from a reachable position by placing a token in the lowest nonempty cell of a column. Because placing this token preserves color balance, it is of the correct color for its placement to be a valid move.
\end{proof}

We will use this principle to prove the following theorem inductively.

\begin{theorem}
\label{thm:reachability}
For all color-balanced positions $P$ of a given size $k$, if the number $i$ of nonempty columns satisfies $i > k/2$, then $P$ is reachable.
\end{theorem}

\begin{proof}
For $k = 1$, any color-balanced position comprises a single red token and is reachable; for $k = 2$, any color-balanced position with $i = 2$ nonempty columns comprises one token of each color on the bottom row of the board and is reachable. Suppose the theorem statement holds for sizes $k - 1$ and $k - 2$, and consider a color-balanced position of size $k$ with $i > k/2$ nonempty columns. 

If $i = (k + 1)/2$, then $k$ is odd, and there are $(k - 1)/2$ yellow tokens; since the crown contains $i$ tokens, it must contain a red token, whose removal would preserve color balance. If $i > (k + 1)/2$, then the crown contains tokens of both colors, since there are at most $(k + 1)/2$ tokens of a given color. In all cases, the crown contains a token of the necessary color such that its removal would preserve color balance. In the remainder of the proof we suppose this color to be red without loss of generality; the subsequent argument is symmetric with respect to color.

Select some column that contains more than one token, and call its topmost token $t$. (If no such column exists, the position is trivially reachable.) 

If $t$ is red, removing it yields a color-balanced position of size $k - 1$ with $i$ nonempty columns, which is reachable by the inductive hypothesis; the original position is then reachable by Lemma~\ref{lem:removal}.

If $t$ is yellow, removing any red crown token (the existence of which is established above) yields a color-balanced position $P'$ of size $k - 1$ with either $i$ or $i - 1$ nonempty columns. Removing $t$ from $P'$ yields a color-balanced position of size $k - 2$ with $i$ or $i - 1$ nonempty columns, which is reachable by the inductive hypothesis. By Lemma~\ref{lem:removal}, the intermediate position $P'$ is reachable; by a second application of Lemma~\ref{lem:removal}, the original position is also reachable.

\end{proof}

\begin{corollary}
\label{cor:partial-T-formula}
If $i > k/2$, then \[T_{k, i} = \binom{k-1}{i-1}\binom{k}{\lceil\frac{k}{2}\rceil}.\]
\end{corollary}

\begin{proof}
 The first factor counts ways to distribute $k$ tokens among $i$ columns with at least one token in each column (i.e., compositions of $k$ into $i$ positive parts); the second factor counts ways to color $\lceil k/2 \rceil$ of the tokens red.
\end{proof}

\begin{corollary}
For $k \geq 2$, the number of color-balanced positions of size $k$ on an $n$-column board that are \emph{not} reachable is given by a polynomial in $n$ of degree $\lfloor k/2 \rfloor$.
\end{corollary}
\begin{proof}
Applying Vandermonde's convolution to the right side of Equation~\eqref{eq:color-balanced} and subtracting the right side of Equation~\eqref{eq:F-expansion} gives \[\sum_{i = 1}^{k}\left(\binom{k-1}{i-1}\binom{k}{\lceil\frac{k}{2}\rceil} - T_{k, i}\right)\binom{n}{i}\] for the number of unreachable color-balanced positions.
This is a polynomial in $n$; by Corollary~\ref{cor:partial-T-formula} its degree is at most $\lfloor k/2 \rfloor$. Its degree is \emph{precisely} $\lfloor k/2 \rfloor$ if $T_{k,i} \neq \binom{k-1}{i-1}\binom{k}{\lceil k/2 \rceil}$ for $i = \lfloor k/2 \rfloor$, that is, if there exist unreachable color-balanced positions of size $k$ with $\lfloor k/2 \rfloor$ nonempty columns.

For $k \geq 2$, such a position may be constructed as follows: distribute  $\lfloor k/2 \rfloor$ yellow tokens arbitrarily among the bottom-row cells of the board, then distribute $\lceil k/2 \rceil$ red tokens arbitrarily (in any manner preserving column-wise vertical contiguity) above these yellow tokens. This yields a color-balanced position with no red token on the bottom row; any reachable position must have a red token on the bottom row, so this position is unreachable.

\end{proof}
\begin{corollary}
For each $k$, as $n \to \infty$, \[F(n, k) \sim \frac{1}{k!}\binom{k}{\lceil\frac{k}{2}\rceil}n^k = \frac{1}{\lceil\frac{k}{2}\rceil!\lfloor\frac{k}{2}\rfloor!}n^k.\]

\end{corollary}
\begin{proof}
This follows from Equation~\eqref{eq:F-expansion} and Corollary~\ref{cor:partial-T-formula}.
\end{proof}

\begin{remark} 
As the number of columns goes to infinity, almost all color-balanced positions of a given size become reachable. This is unsurprising: if columns greatly outnumber tokens, then most positions have tokens sparsely distributed with no stacking. However, our results refine this intuition by quantifying \emph{how much} stacking can occur without losing guaranteed reachability, thereby showing that unreachable color-balanced position counts in fact grow with only \emph{half} the degree of overall color-balanced position counts.
\end{remark}

\section{Behavior of \texorpdfstring{$F(n, k)$}{F(n, k)} for Fixed \texorpdfstring{$n$}{n}}

While our partial formula for $T_{k, i}$ tells us what happens to $F(n, k)$ as $n \to \infty$, it tells us nothing about what happens as $k \to \infty$; for sufficiently large $k$, all terms with $i > k/2$ in Equation~\eqref{eq:F-expansion} vanish due to the factor of $\binom{n}{i}$. In this section, we will establish some basic facts about the exponential growth rate $\lim_{k \to \infty} F(n, k)^{1/k}$; in addition, we will show that a subset of the reachable state space has a natural algebraic interpretation.

If a column's state is considered a sequence of tokens read from bottom to top, then define the addition of positions by column-wise concatenation (intuitively, ``stacking''). Note that the set of positions on an $n$-column board under addition is the $n$th direct power of the free monoid $F_2$ on two generators $R$ and $Y$.

In proving the next theorem, it will be useful to define the \emph{color inversion} $\overline{P}$ of a position $P$ as the position obtained by inverting the color of every token in $P$.

\begin{theorem}
$\lambda_n = \lim_{k \to \infty} F(n, k)^{1/k}$ exists and is at most $2$ for all $n$.
\end{theorem}

\begin{proof}

Consider any two positive integers $a$ and $b$. For each reachable position $P$ of size $a + b$, choose a sequence of moves $S_P$ that realizes $P$, where each move consists of placing a red token (for odd-indexed moves) or a yellow token (for even-indexed moves) in some column. Let $A_P$ be the position realized by the length-$a$ prefix of $S_P$, and $B_P$ be the position realized by the length-$b$ suffix $S'_P$ of $S_P$. Then for all $P$, $A_P$ is reachable and $A_P + B_P = P$. 

If $a$ is even, $S'_P$ always begins with the placement of a red token. Any color-alternating move sequence beginning with the placement of a red token leads to a reachable position, so $B_P$ is reachable. If $a$ is odd, on the other hand, then $S'_P$ always begins with the placement of a yellow token. In this case, $\overline{B_P}$ is reachable.

If $a$ is even, consider $F_1(P) = (A_P, B_P)$; if $a$ is odd, consider $F_2(P) = (A_P, \overline{B_P})$. 

Let $F_0(A, B) = (A, \overline{B})$; then $F_2 = F_0 \circ F_1$. Because $A_P + B_P = P$, the map $F_1$ is injective. Because color inversion is injective, so is the map $F_0$, and therefore, so is $F_2$.

In all cases, there exists an injection from reachable positions of size $a + b$ to ordered pairs of reachable positions of sizes $a$ and $b$. From this, it follows that $F(n, a + b) \leq F(n, a) \cdot F(n, b)$, which implies via Fekete's lemma that $\lambda_n$ exists and equals $\inf_{k \in \mathbb{N}} F(n, k)^{1/k}$.

As shown in Section~\ref{sec:upper-bound}, $F(n, k) \leq \binom{n+k-1}{k}2^k$; as $\binom{n+k-1}{k}$ grows polynomially in $k$, the $k$th root of this upper bound tends to $2$ as $k \to \infty$, from which it follows that $\lambda_n \leq 2$.

\end{proof}

While the following proposition is useful for studying $\lambda_n$, it is worth noting in its own right, because it implies that reachable positions of even size on an $n$-column board form a submonoid of $(F_2)^n$. 

\begin{prop}
\label{prop:even}
Reachable positions of even size are closed under addition.
\end{prop}

\begin{proof}
Let $P_1$ and $P_2$ be reachable positions of even size. Since position size is additive, $P_1 + P_2$ is also of even size. If $S_1$ and $S_2$ are valid move sequences respectively realizing $P_1$ and $P_2$, then their concatenation $S_1S_2$ realizes $P_1 + P_2$. Since $S_1$ and $S_2$ are both color-alternating, begin with the placement of a red token, and end with the placement of a yellow token, $S_1S_2$ also has these properties and is a valid move sequence; therefore, $P_1 + P_2$ is reachable. 
\end{proof}

(Note that reachable positions in general fail to be closed under addition because, for instance, if both summands are of odd size, their sum will contain two more red tokens than yellow tokens and fail to be color-balanced.)

Growth rates of finitely generated monoids and semigroups, both exponential and subexponential, have been studied extensively (see \cite{Trofimov}, \cite{Burillo-Guba}, \cite{Nathanson}, \cite{Lavrik-Männlin}, \cite{Kambites-et-al}). The monoid of even-sized reachable positions on an $n$-column board is finitely generated by the $n^2$ elements corresponding to all possible red-yellow move pairs. Specifically, the generating set is the union of the orbits of $(R, Y, \varepsilon, \dots, \varepsilon)$ and $(RY, \varepsilon, \dots, \varepsilon)$ under the action of the symmetric group $S_n$. Since there are $F(n, 2m)$ length-$m$ products of these generators, $\lambda_n$ can be understood as the square root of the exponential growth rate of this monoid. 

\begin{theorem}
\label{thm:growth-rate} $\lim_{n \to \infty} \lambda_n = 2$.
\end{theorem}

\begin{proof}
On a $2n$-column board, consider the set $M$ of size-$2n$ color-balanced positions with all tokens on the bottom row. There are $\binom{2n}{n}$ such positions, all of which are reachable. In any sum of these positions, the $i$th row from the bottom is identical to the bottom row of the $i$th summand, and hence the sequence of summands is uniquely reconstructible; as such, the monoid generated by $M$ is free on $M$.

By Proposition~\ref{prop:even}, all positions in this monoid are reachable. This gives $F(2n, 2nm) \geq \binom{2n}{n}^{m}$, which implies that $F(2n, k)^{1/k} \geq \binom{2n}{n}^{1/2n}$ infinitely often, and hence $\lambda_{2n} \geq \binom{2n}{n}^{1/2n}$. Since $\lim_{n \to \infty} \binom{2n}{n}^{1/2n} = 2$ and $\lambda_{2n} \leq 2$, it follows that $\lim_{n \to 
\infty} \lambda_{2n} = 2$.

Any reachable position on an $n$-column board is embeddable in an $(n+1)$-column board, so $F(n + 1, k) \geq F(n, k)$. This means that $\lambda_{n + 1} \geq \lambda_n$. Because $\lambda_n$ is bounded above and monotonically increasing in $n$, we conclude that the limit of $\lambda_n$ as $n \to \infty$ exists; because the even subsequence converges to $2$, we conclude that the limit is $2$.

\end{proof}

\begin{remark}
Theorem~\ref{thm:growth-rate} implies that, as $n \to \infty$, the growth rate of the monoid generated by the $S_n$-orbits of $(R, Y, \varepsilon, \dots, \varepsilon)$ and $(RY, \varepsilon, \dots, \varepsilon)$ approaches $4$. Notably, because the proof uses a free monoid generated by a set of positions in which every coordinate is of length $1$, this result also holds for the smaller monoid generated by the $S_n$-orbit of $(R, Y, \varepsilon, \dots, \varepsilon)$ alone. The technique is similar to methods used in \cite{Burillo-Guba} and \cite{Hamann} to lower-bound monoid growth rates.
\end{remark}

\subsection{Numerical Bounds on \texorpdfstring{$\lambda_n$}{lambda n}}

From OEIS entry A320452 \cite{OEIS}, $F(2, 35) = 2,372,700,708$, so $\lambda_2 = \inf_{k \in \mathbb{N}} F(2, k)^{1/k} \leq 2,372,700,708^{1/35} = 1.85295 \dots$ 

Using a variation on the free monoid embedding argument above, we can bound the growth rate more tightly from below. Any set of positions $M$ with exactly $m$ tokens in each column generates a monoid that is free on $M$. (From a sum of such positions, it is always possible to ``read off'' the summands from the bottom up in blocks of height $m$.) For an even number of columns, all positions in this free monoid are reachable.

It has been computationally verified using BFS enumeration that, on a $2$-column board, the number of reachable positions with exactly $16$ tokens in each column is 43,986,734.\footnote{This can also be understood as the number of reachable \emph{final} positions on a $16 \times 2$ board, a distinct combinatorial problem in its own right for boards of arbitrary finite dimensions.}
Therefore, $F(2, 32k) \geq 43,986,734^k$, which implies $\lambda_2 \geq 43,986,734^{1/32} = 1.7332 \dots$ 

Similarly, it has been verified that the number of reachable positions on a $4$-column board with $6$ tokens in each column is $2,421,904$. This bounds $\lambda_4$ from below by $2,421,904^{1/24} = 1.84504 \dots$ 

As $\lambda_n$ grows monotonically, we have the following bounds.

\[1.73322 \dots \leq \lambda_2 \leq 1.85295 \dots\]

\[1.73322 \dots \leq \lambda_3 \leq 2\]

\[1.84504 \dots \leq \lambda_n \leq 2 \qquad (n \geq 4)\]

\begin{remark}
It is unclear whether $\lambda_n$ actually equals $2$ for any particular $n$. The proven facts are consistent with multiple possibilities, of which two seem most likely: either $n = 2$ is a special case and the growth rate is exactly $2$ for $n \geq 3$, or the growth rates merely approach $2$ from below and never reach it (similarly to the $n$-bonacci numbers). For reference, if the ratios of successive terms of $F(3, k)$ are computed from OEIS entry A320731 \cite{OEIS}, the last three are $2.0757883, 2.0713972,$ and $2.0677381$.
\end{remark}

\subsection{Asymptotic Relation Between \texorpdfstring{$F(n,k)$ and $T_{k,n}$}{F(n,k) and T(k,n)}}

\begin{conjecture*}
For fixed $n$, $T_{k, n} \sim F(n, k)$.
\end{conjecture*}

This is equivalent to the following:

\begin{itemize}
 \item The term corresponding to $i = n$ in Equation~\eqref{eq:F-expansion} dominates.

\item $F(n - 1, k) = o(F(n, k))$.

 \item As the number of moves tends to infinity, the proportion of reachable positions with tokens in every column approaches $1$. 

\end{itemize}

In fact, we expect more strongly that for any fixed $m$, as the number of moves tends to infinity, the proportion of reachable positions in which every column contains at least $m$ tokens approaches $1$. If a reachable position is constructed by choosing \emph{moves} uniformly at random, the resulting distribution trivially has this property. While choosing \emph{reachable positions} uniformly at random yields a different distribution, there is no obvious reason to expect that this distribution differs so radically as to falsify the strengthened conjecture; this would require a surprising bias towards positions with a strong imbalance between columns. 

The equivalent form $F(n - 1, k) = o(F(n, k))$ is certainly true if $\lambda_n$ is strictly increasing in $n$, but it is also compatible with $\lambda_n$ being eventually constant, as subexponential factors could still differentiate asymptotic behavior for different values of $n$.

\section{Some Directions for Further Exploration}

\begin{itemize}

\item Some of our results generalize to an arbitrary number of alternating colors; in particular, the proof of Theorem~\ref{thm:polynomiality} applies directly to the general case, and an analogous form of Theorem~\ref{thm:growth-rate} likely holds for more than $2$ colors. One thing that may significantly change is the amount of ``tolerable stacking'' quantified in Theorem~\ref{thm:reachability}.

\item 
We have noted that $F(n, 2m)$ counts length-$m$ elements of the submonoid of $(F_2)^n$ generated by the $S_n$-orbits of $(R, Y, \varepsilon, \dots, \varepsilon)$ and $(RY, \varepsilon, \dots, \varepsilon)$. A natural extension of the generating set is obtained by including the $S_n$-orbit of $(YR, \varepsilon, \dots, \varepsilon)$, making the monoid symmetric with respect to the generators of the underlying free monoid $F_2$. The resulting structure would be a superset of even-sized reachable positions and a subset of even-sized color-balanced positions. This modification has no natural interpretation in the Connect Four setting, but may be of independent algebraic interest.

\item Counting reachable positions up to column permutation is a distinct problem from the one treated here. An interesting property of this problem is that we can allow the number of columns to be infinite, which makes the number of reachable positions a function only of the number of moves. The values of this function on the first eight positive integers, computed by exhaustive enumeration, are $1$, $2$, $5$, $13$, $32$, $81$, $203$, and $518$.

\end{itemize}

\section*{Acknowledgments}
I would like to thank Dr. Gary Gutman for his support and encouragement.

\end{document}